\documentclass[journal]{IEEEtran} 
\IEEEoverridecommandlockouts                              % This command is only
\usepackage{algorithmic,algorithm}
\usepackage{epsfig}
\usepackage{array}
\usepackage{amsmath}
\usepackage{amssymb}
\usepackage{amsthm}
\usepackage{bm}
\usepackage{cite}
\usepackage{color}
\usepackage{balance}
\newtheorem{thm}{Theorem}[section]
\newtheorem{lem}[thm]{Lemma}

\newtheorem{ass}{Assumption}

\newtheorem{prop}[thm]{Proposition}

\newcommand{\argmin}[1]{\underset{#1}{\arg \min}}

\newcommand{\xalgm}{\mathcal{X}_{\rm Alg}}
\newcommand{\xalg}{$\xalgm$ }
\newcommand{\zalgm}{\mathcal{Z}_{\rm Alg}}
\newcommand{\zalg}{$\zalgm$ }

\usepackage{makecell}

\usepackage{xcolor,calc}

\title{Learning continuous $Q$-functions using generalized Benders cuts}
\author{}
\author{Joseph Warrington,~\IEEEmembership{Member,~IEEE}
\thanks{The author is with the Automatic Control Laboratory, Swiss Federal Institute of Technology (ETH) Zurich,
Physikstrasse 3, 8092 Zurich, Switzerland. Contact: {\tt\small warrington@control.ee.ethz.ch}}}

\begin{document}
\maketitle

\thispagestyle{empty}
\begin{abstract}
$\bm{Q}$-functions are widely used in discrete-time learning and control to model future costs arising from a given control policy, when the initial state and input are given. Although some of their properties are understood, $\bm{Q}$-functions generating optimal policies for continuous problems are usually hard to compute. Even when a system model is available, optimal control is generally difficult to achieve except in rare cases where an analytical solution happens to exist, or an explicit exact solution can be computed. It is typically necessary to discretize the state and action spaces, or parameterize the $\bm{Q}$-function with a basis that can be hard to select \textit{a priori}. This paper describes a model-based algorithm based on generalized Benders theory that yields ever-tighter outer-approximations of the optimal $\bm{Q}$-function. Under a strong duality assumption, we prove that the algorithm yields an arbitrarily small Bellman optimality error at any finite number of arbitrary points in the state-input space, in finite iterations. Under additional assumptions, the same guarantee holds when the inputs are determined online by the algorithm's updating $\bm{Q}$-function. We demonstrate these properties numerically on scalar and 8-dimensional systems.
\end{abstract}
\vspace{-0.25cm}

%%%%%%%%%%%%%%%%%%%%%%%%%%%%%%%%%%%%%%%%%%%%%%%%%%%%%%%%%%%%%%%%%

\section{Introduction}

Reinforcement learning (RL) and approximate dynamic programming (ADP) commonly employ so-called $Q$-functions to model the costs incurred in the future evolution of a discrete-time system under a given control policy. The $Q$-function associated with control policy $u = \pi(x)$ takes a state $\hat{x}$ and input $\hat{u}$ as parameters, and is equal to the stage costs incurred immediately for $(\hat{x}, \hat{u})$ plus the costs (typically infinite-horizon with a discount factor) of following policy $\pi$ thereafter.

$Q$-functions are widely associated with RL (i.e., model-free learning), thanks to work stemming from Watkins' $Q$-learning algorithm \cite{watkins_learning_1989}. However, model-free $Q$-learning suffers from slow convergence, even despite new insights into optimizing the rate \cite{devraj_zap_2017}. New work such as \cite{mania_simple_2018} is an example of interest in cases where model data, known or itself learned, can improve learning performance for difficult control problems. The present paper is motivated by a desire to learn an approximate $Q$-function to control a system with a known model.

Mathematically, $Q$-functions have much in common with \emph{value} (or $V$-) functions, the chief difference being that they are defined on state-input space rather than on the state alone. Although they are generally more expensive to store, their higher-dimensional domain often makes approximate, finitely-parameterized $Q$-functions more expressive than $V$-functions \cite[Ch.~2]{busoniu_reinforcement_2010}.

It is common to discretize continuous problems in order to obtain a finite parameterization of the $V$- or $Q$-function \cite{chow_optimal_1991}. However, performing even one iteration of the canonical algorithms, such as value iteration, then has an undesirable exponential cost. ADP methods have arisen to find more tractable parameterizations of the continuous $V$-function. Several are based on continuous extensions of the ``linear programming approach'' to ADP \cite{de_farias_linear_2003}, in which a valid lower bound on the optimal value function is maximized. Examples include the quadratic lower bound in \cite{wang_performance_2009}, and the polynomial derived using sum-of-squares techniques in \cite{summers_approximate_2013}. Approximate $V$-functions represented as the pointwise maximum of multiple lower-bounding functions have been used in \cite{lincoln_relaxing_2006, wang_approximate_2015, beuchat_point-wise_2017, hohmann_moment_2018}. Recent work utilizing a point-wise maximum representation \cite{warrington_generalized_2017} has extended the Benders decomposition argument used for linear multi-stage decision problems in Dual DP (DDP, \cite{pereira_multi-stage_1991}), to a general nonlinear, infinite-horizon setting.

In this paper we adapt the Benders approach from \cite{warrington_generalized_2017} to learn $Q$-functions. We define an algorithm that successively produces tighter approximations of a problem's optimal $Q$-function from below, and prove convergence results for off-policy and policy-driven learning of the $Q$-function in this manner. In the former case, $(x,u)$ pairs are pre-selected at the start of the algorithm, whereas in the latter case, only the $x$ points are pre-selected, and the $u$ decisions are made according to a policy from the update $Q$-function estimate. We then demonstrate the method's efficacy for test systems.

Section \ref{sec:ProblemStatement} describes the infinite horizon problem, Section \ref{sec:Benders} describes the Benders decomposition approach, and Section \ref{sec:Algorithm} proposes an algorithm and proves its key properties. Section \ref{sec:Numerical} presents numerical examples, and Section \ref{sec:Conclusion} concludes.

%%%%%%%%%%%%%%%%%%%%%%%%%%%%%%%%%%%%%%%%%%%%%%%%%%%%%%%%%%%%%%%%%

\section{Problem statement} \label{sec:ProblemStatement}

\subsection{Infinite-horizon control problem} \label{sec:IHProblem}

The scope considered is the class of infinite-horizon, discrete-time, deterministic optimal control problems with time-invariant stage cost functions, dynamics, and constraints:
\begin{subequations} \label{eq:IHProblem}
\begin{align}
V^\star(x) := \inf_{u_0,u_1,\ldots} \quad & \sum_{t=0}^{\infty} \gamma^t \ell(x_t, u_t) \label{eq:IHObj}\\
\text{s.~t.} \quad & x_{t+1} = f(x_t, u_t), \quad t = 0,1,\ldots \, , \\
& h(x_t, u_t) \leq 0, \quad t = 0,1,\ldots \, , \label{eq:IHSIC} \\
& x_0 = x \, .
\end{align}
\end{subequations}
For each time step $t$ we denote the state $x_t \in \mathcal{X} \subseteq \mathbb{R}^{n_x}$, and the action, or input, $u_t \in \mathcal{U} \subseteq \mathbb{R}^{n_u}$. Sets $\mathcal{X}$ and $\mathcal{U}$ are the state and action spaces, and are continuous. Future costs are discounted according to a discount factor $\gamma \in \left(0, 1\right]$, the (non-negative) stage cost function is $\ell : \mathcal{X} \times \mathcal{U} \rightarrow \mathbb{R}_+$, and the dynamics are governed by the mapping $f : \mathcal{X} \times \mathcal{U} \rightarrow \mathcal{X}$. There are $n_c$ state-input constraints \eqref{eq:IHSIC}, parameterized by a vector-valued mapping $h : \mathcal{X} \times \mathcal{U} \rightarrow \mathbb{R}^{n_c}$.  The parametric infimum $V^\star(x)$ of problem \eqref{eq:IHProblem} is referred to as the \emph{optimal value function} (or optimal $V$-function) of the problem.

\subsection{$Q$-functions}

We now define $Q$-functions and briefly state some of their well-known properties for later use. For more detail, see for example \cite[Chapter 2]{busoniu_reinforcement_2010}. Given a policy $\pi : \mathcal{X} \rightarrow \mathcal{U}$, its associated $Q$-function, $Q^\pi : \mathcal{X} \times \mathcal{U} \rightarrow \mathbb{R} \cup \{+\infty\}$, is
\begin{equation} \label{eq:Qdef}
Q^\pi(x,u) = \ell(x,u) + \sum_{t=1}^\infty \gamma^t \ell(x_t, \pi(x_t)) \, ,
\end{equation}
in which the relation $x_{t+1} = f(x_t,\pi(x_t))$ holds for $t \geq 1$, and $x_1 = f(x,u)$. The $Q$-function is the sum of the stage cost incurred for some initial state and input $x$ and $u$, and the infinite sum of (discounted) costs under policy $\pi$ thereafter.

The optimal $Q$-function, which we denote $Q^\star$, minimizes \eqref{eq:Qdef} over policies $\pi$, and satisfies
\begin{equation}\label{eq:optimalQ}
Q^\star(x,u) = \ell(x,u) + \inf_{u' \in \mathcal{U}(f(x,u))} Q^\star(f(x,u), u')
\end{equation}
for all $(x,u) \in \mathcal{X} \times \mathcal{U}$. In \eqref{eq:optimalQ} we use the notation $\mathcal{U}(x) := \{u \in \mathcal{U} \, : \, h(x,u) \leq 0 \}$. An associated Bellman operator for $Q$-functions, $\mathcal{T_Q}$, can be defined as 
\begin{equation} \label{eq:TQdef}
\mathcal{T_Q}Q(x,u) := \ell(x,u) + \inf_{u' \in \mathcal{U}(f(x,u))} Q(f(x,u), u') \, .
\end{equation}
On the left-hand side, $\mathcal{T_Q}Q$ is to be interpreted as a new function with the same domain as $Q$, and evaluated at $(x,u)$. Thus, condition \eqref{eq:optimalQ} can be written $\mathcal{T_Q}Q^\star(x,u) = Q^\star(x,u)$ for all $(x,u) \in \mathcal{X} \times \mathcal{U}$. If an optimal $Q$- and $V$-function exist for problem \eqref{eq:IHProblem}, they are related by 
$V^\star(x) = \inf_{u \in \mathcal{U}(x)} Q^\star(x,u)$, and thus from \eqref{eq:optimalQ}, $Q^\star(x,u) = \ell(x,u) + \gamma V^\star(f(x,u))$.

For any approximate $Q$-function for which the infimum in \eqref{eq:TQdef} is attained, one can define an associated control policy consistent with definition \eqref{eq:Qdef}:
\begin{equation}\label{eq:Qpolicy}
\pi(x; Q) \in \argmin{u \in \mathcal{U}(x)}\, Q(x,u) \, .
\end{equation}
The attraction of a $Q$-function is that in a wide range of cases it is simpler to solve \eqref{eq:Qpolicy} than it would be to solve, for the same $x$, the full infinite-horizon problem \eqref{eq:IHProblem}, or a finite-horizon truncation thereof, as in Model Predictive Control (MPC) \cite{borrelli_predictive_2017}. 

Lastly, for the benefit of developments in Section \ref{sec:Benders}, we note it is easy to show that the operator $\mathcal{T_Q}$ is monotonic:
\begin{align}
& Q_a(x,u) \leq Q_b(x,u) \quad \forall (x,u) \in \mathcal{X} \times \mathcal{U} \nonumber \\
& \hspace{0.6cm} \Rightarrow \mathcal{T_Q}Q_a(x,u) \leq \mathcal{T_Q}Q_b(x,u) \quad \forall (x,u) \in \mathcal{X} \times \mathcal{U} \, . \label{eq:QMonotonic}
\end{align}

\section{Benders cuts} \label{sec:Benders}

\subsection{Pointwise maximum representation}

Let $Q_I : \mathcal{X} \times \mathcal{U} \rightarrow \mathbb{R}$ be a function of the following ``pointwise maximum'' form,
\begin{equation} \label{eq:pwmax}
Q_I(x,u) = \max_{i=0,\ldots,I} \{ q_i(x,u) \} \, ,
\end{equation}
where $I$ is a non-negative integer, and each function $q_i : \mathcal{X} \times \mathcal{U} \rightarrow \mathbb{R}$ is known to satisfy 
\[q_i(x,u) \leq Q^\star(x,u) \, , \quad \forall(x,u) \in \mathcal{X} \times \mathcal{U}\, .\] 
Thus $\smash{Q_I(x,u) \leq Q^\star(x,u)}$ for all $\smash{(x,u) \in \mathcal{X} \times \mathcal{U}}$. From \eqref{eq:Qpolicy} the control policy associated with $Q_I$ is simply $\pi(x;Q_I) \in \arg\min_{u\in\mathcal{U}(x)} \max_{i=0,\ldots,I} \{q_i(x,u)\}$.

In Section \ref{sec:Algorithm} we will propose an algorithm that uses $Q_I$ to construct an additional function, or ``cut'' $q_{I+1}$. Under certain assumptions, the new cut satisfies
\begin{subequations}
\begin{align}
q_{I+1}(x,u) & \leq Q^\star(x,u) \, \,\, \forall (x,u) \in \mathcal{X} \times \mathcal{U} \, , \label{eq:underest} \\
\text{and } q_{I+1}(\hat{x},\hat{u}) & > Q_I(\hat{x},\hat{u}) \, \,\,  \text{for some }(\hat{x},\hat{u}) \in \mathcal{X} \times \mathcal{U}. \label{eq:improvement}
\end{align}
\end{subequations}
Thus the new function, $\smash{Q_{I+1}(x,u) := \max \{Q_I(x,u),}$ $q_{I+1}(x,u)\}$, will be a tighter under-approximation of $Q^\star$ than $Q_I$. We now derive a Benders-type procedure to achieve this, which is related to that in \cite{warrington_generalized_2017} for $V$-functions. 
\vspace{-0.3cm}

\subsection{Duality in operator $\mathcal{T_Q}$}

We start by taking the dual of the minimization problem solved inside the operator $\mathcal{T_Q}$ at some point $(\hat{x},\hat{u})$ in the state-action space. For a function $Q_I$ taking the form \eqref{eq:pwmax}, the right-hand side of \eqref{eq:TQdef} can be written equivalently as
\begin{align*}
\mathcal{T_Q}Q_I(\hat{x},\hat{u}) = \inf_{x', u'} \quad & \ell(\hat{x},\hat{u}) + \gamma \max_{i=0,\ldots,I} \{ q_i(x',u') \} \\
\text{s.~t.}\quad & x' = f(\hat{x},\hat{u}) \, , \\
& h(x',u') \leq 0 \, ,
\end{align*}
where the extra variable $x' \in \mathbb{R}^{n_x}$ is introduced to model the successor state explicitly. An epigraph variable $\alpha \in \mathbb{R}$ can be introduced to replace the inconvenient maximum operator in the objective with $I+1$ separate constraints. This leads to an equivalent problem:
\begin{subequations} \label{eq:osp_primal}
\begin{align}
\mathcal{T_Q}Q_I(\hat{x},\hat{u}) = \inf_{x', u',\alpha} \quad & \ell(\hat{x},\hat{u}) + \gamma \alpha \label{eq:osp_obj}\\
\text{s.~t.}\quad & x' = f(\hat{x},\hat{u}) \, , \label{eq:osp_dyn}\\
& h(x',u') \leq 0 \, , \label{eq:osp_sic}\\
& q_i(x', u') \leq \alpha \, , \quad i=0,\ldots,I \, . \label{eq:osp_epi}
\end{align}
\end{subequations}

Assigning the Lagrange multipliers $\nu \in \mathbb{R}^{n_x}$, $\lambda_c \in \mathbb{R}^{n_c}_+$, and $\lambda_\alpha \in \mathbb{R}^{I+1}_+$ to constraints \eqref{eq:osp_dyn}, \eqref{eq:osp_sic}, and \eqref{eq:osp_epi} respectively, one can form the Lagrangian,
\begin{align*}
&\mathcal{L}(x',u',\alpha,\nu,\lambda_c, \lambda_\alpha)  := \ell(\hat{x},\hat{u}) + \gamma \alpha + \nu^\top\!(f(\hat{x},\hat{u}) - x') \\
& \hspace{2.5cm}  + \lambda_c^\top h(x'\!, u')  + \sum_{i=0}^I \lambda_{\alpha,i} (q_i(x',u') - \alpha) \, .
\end{align*}
Following standard procedure, the dual of \eqref{eq:osp_primal} is then
\begin{subequations} \label{eq:osp_dual}
\begin{align}
J_D(\hat{x},\hat{u}) := \sup_{\nu, \lambda_c, \lambda_\alpha} \,\, & \ell(\hat{x},\hat{u}) + \nu^\top f(\hat{x},\hat{u}) + \xi(\nu, \lambda_c, \lambda_\alpha) \\
\text{s.~t.}\quad & \mathbf{1}^\top \lambda_\alpha = \gamma \, , \\
& \lambda_c \geq 0 \, , \quad \lambda_\alpha \geq 0 \, ,
\end{align}
\end{subequations}
where the function
\begin{equation*}
\xi(\nu, \lambda_c, \lambda_\alpha) := \inf_{x'\!,u'} \!\! \left\{ \! -\nu^\top \!x' + \lambda_c^\top h(x'\!,u') \!+ \! \sum_{i=0}^{I} \lambda_{\alpha,i} q_i(x'\!,u') \! \right\}
\end{equation*}
depends only on the multipliers. Although problem \eqref{eq:osp_primal} may not be convex, the objective of \eqref{eq:osp_dual} is always concave \cite[\S 5.2]{boyd_convex_2009}, and weak duality implies $J_D(\hat{x},\hat{u}) \leq \mathcal{T_Q}Q_I(\hat{x}, \hat{u})$ for any choice of parameter $(\hat{x},\hat{u}) \in \mathcal{X} \times \mathcal{U}$. 

\subsection{Generalized Benders cut}

Given a function $Q_I$ of the form \eqref{eq:pwmax} such that $Q_I \leq Q^\star$, suppose that optimal multipliers $(\hat{\nu}^\star, \hat{\lambda}_c^\star, \hat{\lambda}_\alpha^\star)$ are attained when \eqref{eq:osp_dual} is solved with parameter $(\hat{x},\hat{u})$. These can be used to form a new cut $q_{I+1}(\cdot,\cdot)$ with the following attractive properties.
\begin{lem} \label{lem:new_lb}
The function 
\begin{equation}
q_{I+1}(x,u) := \ell(x,u) + \hat{\nu}^{\star\, \top} f(x,u) + \xi(\hat{\nu}^\star,\hat{\lambda}_c^\star,\hat{\lambda}_\alpha^\star) \label{eq:newlb_def}
\end{equation}
satisfies $q_{I+1}(x,u) \leq Q^\star(x,u)$ for all $(x,u) \in \mathcal{X} \times \mathcal{U}$.
\end{lem}
\begin{proof}
%Given $(\hat{\nu}^\star,\hat{\lambda}_c^\star,\hat{\lambda}_\alpha^\star)$, the solution to the dual problem \eqref{eq:osp_dual} with parameter $(\hat{x}, \hat{u})$, consider the (attained) solution to the same problem but for another parameter $(\overline{x},\overline{u})$. Call this $(\overline{\nu}^\star, \overline{\lambda}_c^\star, \overline{\lambda}_\alpha^\star)$. Then we have
%\begin{align*}
%J_D(\overline{x},\overline{u}) = \ell(\overline{x},\overline{u}) + \overline{\nu}^{\star T}f(\overline{x},\overline{u}) + \xi(\overline{\nu}^{\star}, \overline{\lambda}_c^\star, \overline{\lambda}_\alpha^\star) \, .
%\end{align*}
%If the dual supremum at $(\overline{x},\overline{u})$ is not attained, one can simply use the full definition of $J_D(\overline{x},\overline{u})$ in \eqref{eq:osp_dual}.

An optimal dual solution $(\hat{\nu}^\star,\hat{\lambda}_c^\star,\hat{\lambda}_\alpha^\star)$ for parameter $(\hat{x}, \hat{u})$ must in general be a suboptimal solution to problem \eqref{eq:osp_dual} when any other parameter $(\overline{x}, \overline{u}) \in \mathcal{X} \times \mathcal{U}$ is used, i.e.,
\begin{equation*}
\ell(\overline{x},\overline{u}) + \hat{\nu}^{\star T}f(\overline{x},\overline{u}) + \xi(\hat{\nu}^{\star}, \hat{\lambda}_c^\star, \hat{\lambda}_\alpha^\star) \leq J_D(\overline{x},\overline{u}) \, .
\end{equation*}
Note that $(\hat{\nu}^\star,\hat{\lambda}_c^\star,\hat{\lambda}_\alpha^\star)$ is feasible in \eqref{eq:osp_dual} for all parameters $(\overline{x},\overline{u})$, as the feasible set is independent of the parameter. From weak duality, $J_D(\overline{x},\overline{u}) \leq \mathcal{T_Q}Q_I(\overline{x},\overline{u})$. As we start with $Q_I \leq Q^\star$ on its domain, we have from the mononoticity property \eqref{eq:QMonotonic} that $\mathcal{T_Q}Q_I(\overline{x},\overline{u}) \leq \mathcal{T_Q}Q^\star(\overline{x},\overline{u})$, and the Bellman optimality condition states that $\mathcal{T_Q}Q^\star(\overline{x},\overline{u}) = Q^\star(\overline{x},\overline{u})$. Combining these relationships we obtain \[ \ell(\overline{x},\overline{u}) + \hat{\nu}^{\star T}f(\overline{x},\overline{u}) + \xi(\hat{\nu}^{\star}, \hat{\lambda}_c^\star, \hat{\lambda}_\alpha^\star) \leq Q^\star(\overline{x},\overline{u}) \, , \] and the result follows simply by noting that $(\overline{x},\overline{u})$ can refer to any $(x,u) \in \mathcal{X} \times \mathcal{U}$ in the argument above. 
\end{proof}

This proof leverages the (generalized) Benders decomposition argument, which was first developed in \cite{geoffrion_generalized_1972} to partition a two-stage problem into two subproblems linked by an approximate value function. Here we have used the properties of $Q$-functions to accommodate the infinite number of stages in problem \eqref{eq:IHProblem}. A similar result was derived for $V$-functions in \cite{warrington_generalized_2017}.

The following properties concern the violation of the Bellman optimality condition \eqref{eq:optimalQ}, or the \emph{$Q$-Bellman error}:\begin{equation} \label{eq:QBellman}
\varepsilon(x, u; Q_I) := \mathcal{T_Q}Q_I(x, u) - Q_I(x, u) \, .
\end{equation}

\begin{lem} \label{lem:PositiveEps}
If $\varepsilon(x, u; Q_I) \geq 0$ for all $(x,u) \in \mathcal{X} \times \mathcal{U}$, then $\varepsilon(x, u; Q_{I+1}) \geq 0$ for all $(x,u) \in \mathcal{X} \times \mathcal{U}$, where \[Q_{I+1}(\cdot, \cdot) = \max \{q_{I+1}(\cdot, \cdot), Q_I(\cdot, \cdot)\} \, .\]
\end{lem}
\begin{proof}
A simple adaptation of \cite[Lemma III.3]{warrington_generalized_2017}.
\end{proof}

\begin{lem} \label{lem:SDIncrease}
Suppose strong duality holds between problems \eqref{eq:osp_primal} and \eqref{eq:osp_dual} and that $\varepsilon(x, u; Q_I) = \mathcal{T_Q}Q_I(x, u) - Q_I(x, u) \geq 0$ for all $(x,u) \in \mathcal{X} \times \mathcal{U}$. Then if at some $(\hat{x}, \hat{u})$  we have $\mathcal{T_Q}Q_I(\hat{x},\hat{u}) > Q_I(\hat{x},\hat{u})$, a cut there is strictly improving: \[Q_{I+1}(\hat{x},\hat{u}) > Q_I(\hat{x},\hat{u})\, , \] and the increase is equal to $\varepsilon(\hat{x}, \hat{u}; Q_I)$.
\end{lem}
\begin{proof}
If strong duality holds, we have $J_D(\hat{x},\hat{u}) = \mathcal{T_Q}Q_I(x,u)$, and the new function $q_{I+1}$ satisfies $q_{I+1}(\hat{x},\hat{u}) = \mathcal{T_Q}Q_I(\hat{x},\hat{u})$. Since $Q_{I+1}(x,u) := \max \{Q_I(x,u),q_{I+1}(x,u)\}$ the result follows.
\end{proof}
Lastly, the following property facilitates a ``greedy'' cut $q_{I+1}(\cdot,\cdot)$ with respect to some particular $(\hat{x},\hat{u})$ location.
\begin{lem} \label{lem:max_gain}
The Benders cut that yields the greatest increase at $(\hat{x},\hat{u})$, i.e., for which $Q_{I+1}(\hat{x},\hat{u}) - Q_I(\hat{x},\hat{u})$ is maximized, is that obtained by solving problem \eqref{eq:osp_dual} at $(\hat{x}, \hat{u})$.
\end{lem}
\begin{proof}
The result follows by reversing the roles of $(\hat{x},\hat{u})$ and $(\overline{x}, \overline{u})$ in the proof of Lemma \ref{lem:new_lb}.
\end{proof}
%As $\xi(\nu,\lambda_c,\lambda_\alpha)$ depends on the lower-bounding functions $q_0,\ldots,q_I$ already generated, Lemma \ref{lem:max_gain} does \emph{not} imply that one can obtain the best possible lower bound on $Q^\star(\hat{x},\hat{u})$ just by ``greedily'' performing repeated cuts at $(\hat{x}, \hat{u})$.

\section{Benders algorithm for Q-Functions} \label{sec:Algorithm}

We propose Algorithm \ref{alg:QBenders} as a means of approximating $Q^\star$ by generating Benders cuts of the form \eqref{eq:newlb_def}. It starts with $q_0 = \ell$, which from \eqref{eq:Qdef} trivially lower-bounds $Q^\star$, and by Lemmas \ref{lem:PositiveEps} and \ref{lem:SDIncrease} guarantees $\varepsilon(x, u; Q_I) \geq 0$ for all $(x,u) \in \mathcal{X} \times \mathcal{U}$ and for all $I \geq 0$. New cuts are created at certain points $(x_m, u_m)$, and Variants A and B differ in how these are chosen:
\begin{itemize}
\item[A.] Select a list of state-input pairs $\zalgm := \{(x_1, u_1),\ldots, $ $(x_M, u_M)\}$ \textit{a priori}, and choose a random $(x_m,u_m)$ at each algorithm iteration.
\item[B.] Select a list of state space points $\xalgm := \{x_1, \ldots, x_M\}$ \textit{a priori}, and within the algorithm pick a random $x_m$, letting $u_m$ follow from policy \eqref{eq:Qpolicy} parameterized by $Q_I$.
\end{itemize}

We now state convergence results for both variants.

\begin{algorithm}[t]
\caption{$Q$-Benders algorithm} 
\label{alg:QBenders}
\begin{algorithmic}[1]
\STATE \textbf{Variant A:} Choose $\zalgm := \{(x_1, u_1), \ldots, (x_M, u_M)\}$
\STATE \textbf{Variant B:} Choose $\xalgm := \{x_1, \ldots, x_M\}$
\STATE Set $I = 0$ and $q_0(x,u) = \ell(x,u)$
\WHILE{TRUE}
\STATE $Q_I(\cdot,\cdot) \gets \max_{i=0,\ldots,I} q_i(\cdot,\cdot)$
\FOR{$m=1,\ldots,M$} \label{algl:BellGapStart}
\STATE \textbf{Variant A:} $(x_m, u_m)$ taken from $\zalgm$ \label{algl:VarAxmum}
\STATE \textbf{Variant B:} $x_m$ taken from $\xalgm$, 
\STATE \hspace{1.62cm} $u_m \gets \arg\min_{u \in \mathcal{U}(x_m)} Q_I(x_m,u)$ \label{algl:VarBum}
\STATE $\varepsilon(x_m, u_m; Q_I) \gets \mathcal{T_Q}Q_I(x_m, u_m) - Q_I(x_m, u_m)$
\ENDFOR \label{algl:BellGapEnd}
\IF{$\max_{m=1,\ldots,M} \{\varepsilon(x_m, u_m; Q_I)\} \leq \varepsilon_\text{tol}$} \label{algl:ConvCheck2}
\STATE \textbf{break}
\ENDIF \label{algl:ConvCheck2_2}
\STATE $m \gets \texttt{UniformRandom}(1,2,\ldots,M)$ \label{algl:Pickx}
\STATE Pick $(x_m, u_m)$ as in lines \ref{algl:VarAxmum}-\ref{algl:VarBum}
\IF{problem \eqref{eq:osp_primal} feasible with parameter $(x_m, u_m)$ \AND dual optimal solution $(\hat{\nu}^\star, \hat{\lambda}_c^\star, \hat{\lambda}_\alpha^\star)$ available}
\STATE Add $q_{I+1}(\cdot,\cdot)$ param'd~by $(\hat{\nu}^\star, \hat{\lambda}_c^\star, \hat{\lambda}_\alpha^\star)$ as in \eqref{eq:newlb_def} \label{algl:AddLB2}
%\ELSE
%\STATE $Q^\star(x_m, u) = +\infty \,\, \forall u \in \mathcal{U}(x_m)$; do not revisit $x_m$
\ENDIF
\STATE $I \gets I+1$
\ENDWHILE
\STATE Return $Q_I(\cdot,\cdot) = \max_{i=0,\ldots,I} q_i(\cdot,\cdot)$
\end{algorithmic}
\end{algorithm}

\subsection{Fixed $(x,u)$ pairs}

The following results hold for Variant A. We omit the proofs of both, because they carry across with little modification from the $V$-function results in \cite[Thms.~III.5 and III.6]{warrington_generalized_2017}:

\begin{thm}[Pointwise convergence of $\{Q_I(x,u)\}_{I=0}^\infty$] \label{thm:QLimEverywhere}
For each $\smash{(x,u) \in \mathcal{X}\times\mathcal{U}}$ for which $Q^\star(x,u)$ is finite, there exists a limiting value $\smash{Q_{\rm lim}(x,u) \leq Q^\star(x,u)}$ such that $\lim_{I \rightarrow \infty} Q_I(x,u) = Q_{\rm lim}(x,u)$.
\end{thm}

\begin{thm}[Finite termination of Variant A] \label{thm:EpsConv}
Suppose the following conditions are met:
\begin{enumerate}
\item[(i)] Strong duality holds for the one-stage problem \eqref{eq:osp_primal} with parameter $(x_m, u_m)$ each time it is solved, for each $(x_m,u_m) \in \zalgm$.
\item[(ii)] $Q^\star(x_m,u_m)$ is finite for each pair $(x_m, u_m) \in \zalgm$.
\end{enumerate}
Then Variant A of Algorithm \ref{alg:QBenders} terminates in finite iterations with probability $1$ for any tolerance $\varepsilon_\text{\normalfont tol} > 0$.
\end{thm}

\subsection{Fixed $x$, policy-driven $u$}

Although Variant A has attractive convergence properties, it learns a $Q$-function based only on performance at pre-selected pairs $(x,u)$, in the sense of minimizing the $Q$-Bellman error there. Variant B instead learns a $Q$-function based on performance at $(x,u)$ pairs in which the $u$ is consistent with the policy derived from the learnt $Q$-function. One expects this criterion to be more relevant to performance of the final policy, as state-input trajectories will pass closer to these points.

Finite termination of Variant B is our main result, which we now state precisely along with the required assumptions.

\begin{ass} \label{ass:FiniteQ}
For each $x_m \in \xalgm$, the set of feasible inputs $\mathcal{U}(x_m)$ contains an element $\hat{u}$ such that $Q^\star(x_m, \hat{u}) < \infty$.
\end{ass}

This assumption implies $V^\star(x_m)$ is finite for each $x_m$. Introducing the notation $\underline{Q}(x_m) := \inf_{u \in \mathcal{U}(x_m)} Q(x_m, u)$, the following holds:
\begin{thm}[Monotone convergence of $\{\underline{Q}_I(x_m)\}_{I=0}^\infty$] \label{thm:Alg2Limit}
Under Assumption \ref{ass:FiniteQ}, the limit $\underline{Q}_{\lim}(x_m) := \lim_{I\rightarrow\infty} \underline{Q}_I(x_m)$ exists for each $x_m \in \xalgm$.
\end{thm}
\begin{proof}
It follows from Assumption \ref{ass:FiniteQ} that $\underline{Q}^\star(x_m) < \infty$, and from Lemma \ref{lem:new_lb}, $Q_I(x_m, u) \leq Q^\star(x_m,u)$ for all $u \in \mathcal{U}(x_m)$. Thus, $\underline{Q}_I(x_m) < \infty$ at each iteration $I$. As the sequence of functions $\{Q_I\}_{I=0}^\infty$ increases monotonically, the sequence $\{\underline{Q}_I(x_m)\}_{I=0}^\infty$ must also increase monotonically. This latter sequence is bounded from above, thus the limit $\lim_{I\rightarrow\infty}\underline{Q}_I(x_m) = \underline{Q}_{\lim}(x_m)$ exists from the Monotone Convergence Theorem.
\end{proof}

An additional performance guarantee for Variant B is available when the following additional assumptions hold.

\begin{ass} \label{ass:CompactU}
For each $x_m \in \xalgm$, set $\mathcal{U}(x_m)$ is compact, and each entry of $f(x_m, u)$ is Lipschitz-continuous on $\mathcal{U}(x_m)$.
\end{ass}

\begin{ass} \label{ass:StrongDuality}
The problem data in \eqref{eq:IHProblem} is such that the lower-bounding functions $q_0, q_1, \ldots$ generated in Variant B:
\begin{itemize}
\item[(i)] Maintain strong duality between problems \eqref{eq:osp_primal}  and \eqref{eq:osp_dual} with parameter $(x_m,u)$ at each iteration of the algorithm, with $u = \pi(x_m; Q_I)$, for all $x_m \in \xalgm$.
\item[(ii)] Are Lipschitz continuous in $u$ with some constant $L_m$ common to all functions $q_i$, for each $x_m \in \xalgm$.
\end{itemize}
\end{ass}
Assumptions \ref{ass:CompactU} and \ref{ass:StrongDuality} must be verified for a given problem. A widespread setting where these hold is the constrained, stable linear-quadratic regulator (LQR); see the Appendix.

\begin{thm}[Finite termination of Variant B] \label{thm:EpsConv2}
Suppose that in addition to Assumption \ref{ass:FiniteQ}, Assumptions \ref{ass:CompactU} and \ref{ass:StrongDuality} hold. Then Variant B of Algorithm \ref{alg:QBenders} terminates in finite iterations with probability $1$ for any tolerance $\varepsilon_\text{\normalfont tol} > 0$.
\end{thm}
\begin{proof}
Let the sequence of iterations $I$ where a given $m$ is chosen in line \ref{algl:Pickx} of the algorithm be indexed by $I_m$. With probability $1$, this sequence is infinitely long for each $m$. We now show that the sequence of $Q$-Bellman errors \[\{\varepsilon(x_m, \pi(x_m; Q_{I_m}); Q_{I_m})\}_{I_m=0}^\infty\] is a Cauchy sequence converging to zero for each $x_m \in \xalgm$. As $\mathcal{U}(x_m)$ is compact for all $x_m \in \xalgm$, the policy $\pi(x_m; Q_I)$ defined in \eqref{eq:Qpolicy} can always be evaluated.

Recall that Lemma \ref{lem:PositiveEps} implies $\varepsilon(x_m, \pi(x_m; Q_{I_m}); $ $Q_{I_m}) \geq 0$ for all $x_m$ and $I_m$. Suppose for the sake of contradiction that the sequence $\{\varepsilon(x_m, \pi(x_m; Q_{I_m}); Q_{I_m})\}_{I_m=0}^\infty$ is \emph{not} a Cauchy sequence converging to $0$. Then there must exist some $\delta > 0$ for which there is no iteration number beyond which $\varepsilon(x_m, \pi(x_m; Q_{I_m}); Q_{I_m}) < \delta$. Whenever point $x_m$ is picked in line \ref{algl:Pickx} of the algorithm, the strong duality condition in Assumption \ref{ass:StrongDuality} and Lemma \ref{lem:SDIncrease} together imply that 
\begin{align*}
Q_{I_m+1}(x_m, \pi(x_m; Q_{I_m})) - Q_{I_m}(x_m, \pi(x_m; Q_{I_m}))\hspace{1.4cm} \\
\hspace{4.6cm} = \varepsilon(x_m, \pi(x_m; Q_{I_m}); Q_{I_m}) \, .
\end{align*}
If $\{\varepsilon(x_m, \pi(x_m; Q_{I_m}); Q_{I_m})\}_{I_m=0}^\infty$ is not a Cauchy sequence, there will be an infinite number of occasions on which \[Q_{I_m+1}(x_m, \pi(x_m; Q_{I_m})) - Q_{I_m}(x_m, \pi(x_m; Q_{I_m})) \geq \delta \, .\] 

Furthermore, Assumption \ref{ass:FiniteQ} and part (ii) of Assumption \ref{ass:StrongDuality} together imply that \[ Q_I(x_m, u) \leq Q^\star(x_m, \hat{u}) + L_m||u - \hat{u}|| \, , \, \forall u \in \mathcal{U}(x_m) \, , \, \forall I \in \mathbb{N}, \] as $Q_I$ is always a lower bound on $Q^\star$. Compactness of $\mathcal{U}(x_m)$ implies that the volume of the truncated hypograph \[ \mathcal{H}_m := \left\{u,s  \left| \!\! \begin{array}{l}Q_I(x_m,u) \leq s \leq Q^\star(x_m, \hat{u}) + L_m||u - \hat{u}|| , \\ u \in \mathcal{U}(x_m) \end{array} \!\!\! \right. \right\}\] is finite for each $x_m$; recall that $q_0(\cdot, \cdot) \equiv \ell(\cdot, \cdot) \geq 0$. 

Due to Lipschitz continuity, cut $q_{I_m + 1}$ decreases the volume of $\mathcal{H}_m$ by an amount that is lower bounded by a function of $\delta$, $L_m$, and the input dimension $n_u$. Thus, this volume cannot be removed infinitely many times from $\mathcal{H}_m$, and we have a contradiction. Cuts made at iterations $I$ where some other index $m' \neq m$ is picked in line \ref{algl:Pickx} may also remove some volume from $\mathcal{H}_m$, but this does not affect the argument. Thus $\{\varepsilon(x_m, \pi(x_m; Q_{I_m}); Q_{I_m})\}_{I_m=0}^\infty$ is a Cauchy sequence converging to zero, and Algorithm \ref{alg:QBenders} terminates in finite iterations for any $\varepsilon_\text{tol} > 0$.
\end{proof}

Therefore, under certain assumptions one need only specify $\xalgm = \{x_1, \ldots, x_M\}$, and Variant B minimizes the $Q$-Bellman error at a $u \in \mathcal{U}(x_m)$ associated with each $x_m \in \xalgm$ that is consistent with policy \eqref{eq:Qpolicy}. One then expects the optimal $Q$-function to be learnt more accurately around the policy surface than elsewhere in the state-action space.

\section{Numerical examples} \label{sec:Numerical}

We now report two numerical tests of Algorithm \ref{alg:QBenders}. In both cases, systems were of the class \textit{\textbf{C-LQR}} described in the Appendix, for which finite termination of Variants A and B is guaranteed by Theorems \ref{thm:EpsConv} and \ref{thm:EpsConv2} respectively, and lower bounding functions $q_i(\cdot,\cdot)$ are quadratic. All tests used the stage cost $\ell(x_t,u_t) = \tfrac{1}{2}x_t^\top x_t + \tfrac{1}{2}u_t^\top u_t$, discount rate $\gamma = 1$, and termination tolerance $\varepsilon_\text{tol} = 10^{-3}$, with $h(x_t,u_t)$ encoding an input constraint $\|u_t\|_\infty \leq 1$. Tests were implemented in Python with subproblems solved using Gurobi 7.0.2, on a computer with an Intel i7 CPU at 2.60 GHz and 16 GB RAM. 

\subsubsection*{Scalar system}

For ease of visualization, we used the simple system $x_{t+1} = 0.9 \, x_t + u_t$, with $x_t,u_t \in \mathbb{R}$, and ran both algorithm variants. In Variant A, \zalg contained 50 random states $x$ sampled uniformly from the interval $[0,3]$, and a random $u \in [-1, 1]$ associated with each $x$. In Variant B the associated inputs were dropped to form $\xalgm$. Fig.~\ref{fig:conv} shows convergence of the maximum $Q$-Bellman error $\max_{x_m \in \xalgm}\varepsilon(x_m, \pi(x; Q_I); Q_I)$, with the mean error $\tfrac{1}{M}\sum_{m=1}^M\varepsilon(x_m, \pi(x_m; Q_I); Q_I)$ shown for comparison. Total time spent generating lower-bounding functions was 243 ms for Variant A and 121 ms for Variant B. For this simple system the optimal policy can be computed as $\pi^\star(x) = \min\{1,\max\{-0.5377x, -1\}\}$. Fig.~\ref{fig:Qplots} shows the evolution of $Q_I$ at selected iterations; after termination, the policy \eqref{eq:Qpolicy} from Variant B was closer to the optimal policy than that from Variant A. The average closed-loop cost starting from points $x \in \xalgm$ was also lower at 2.47991, compared to 2.48716 for Variant A, and 2.47968 for the optimal policy.

\begin{figure}[tb]
\begin{center}
\includegraphics[width=8.9cm]{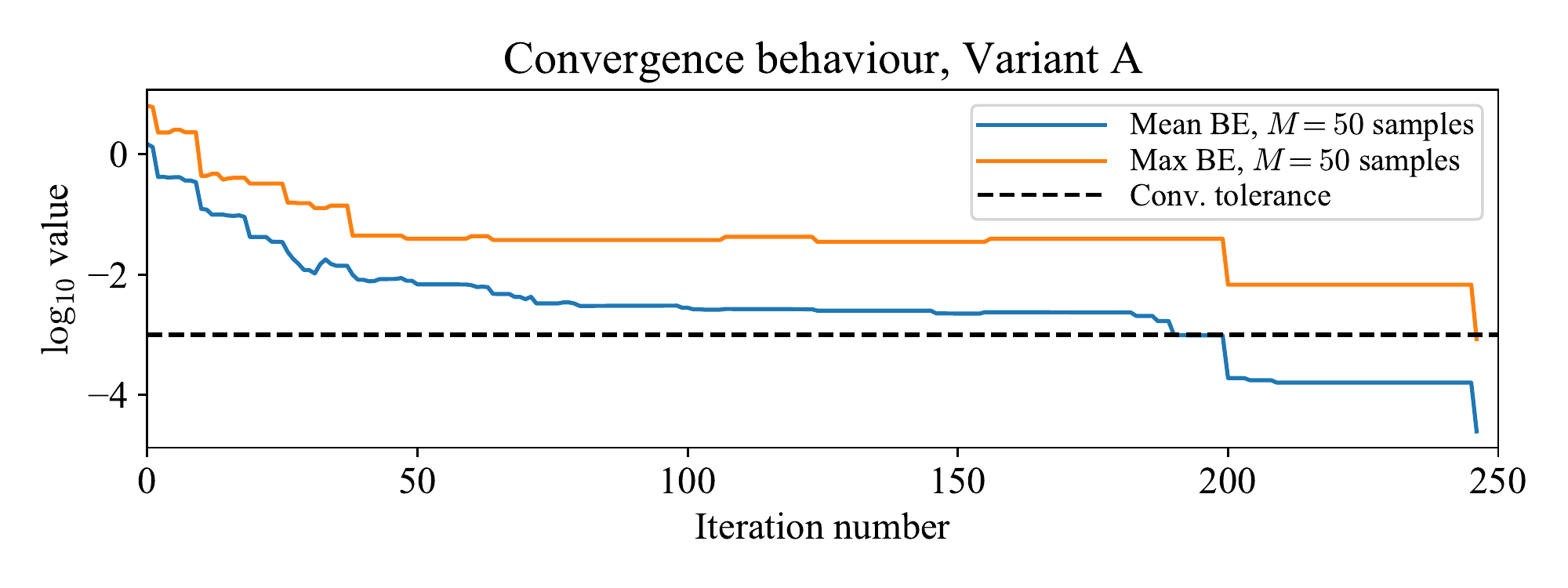}\\[-0.2cm]
\includegraphics[width=8.9cm]{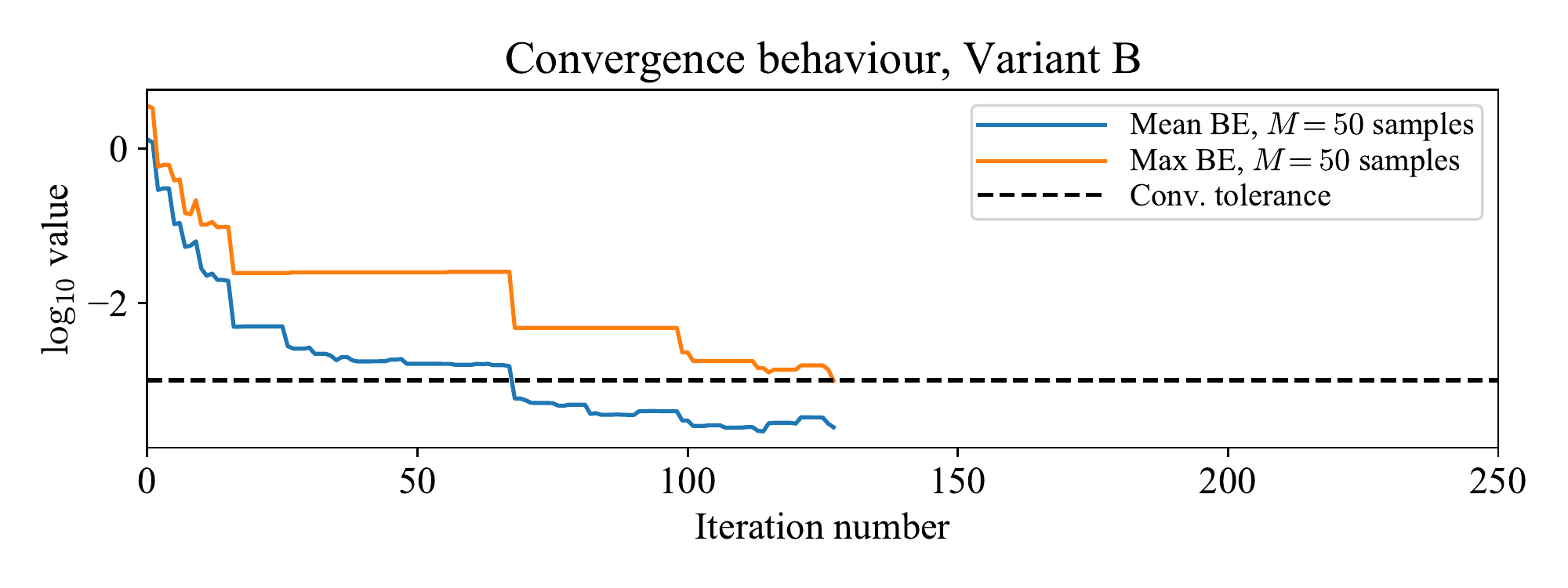}\\
\vspace{-0.7cm}
\end{center}
\caption{Convergence behaviour under Variants A (top) and B (bottom) of Algorithm \ref{alg:QBenders} for a 1-dimensional system. Variant A terminated in 246 iterations, while Variant B terminated in 127 iterations.}
\vspace{-0.4cm}
\label{fig:conv}
\end{figure}

\begin{figure*}[tb]
\begin{center}
\includegraphics[width=0.245\textwidth]{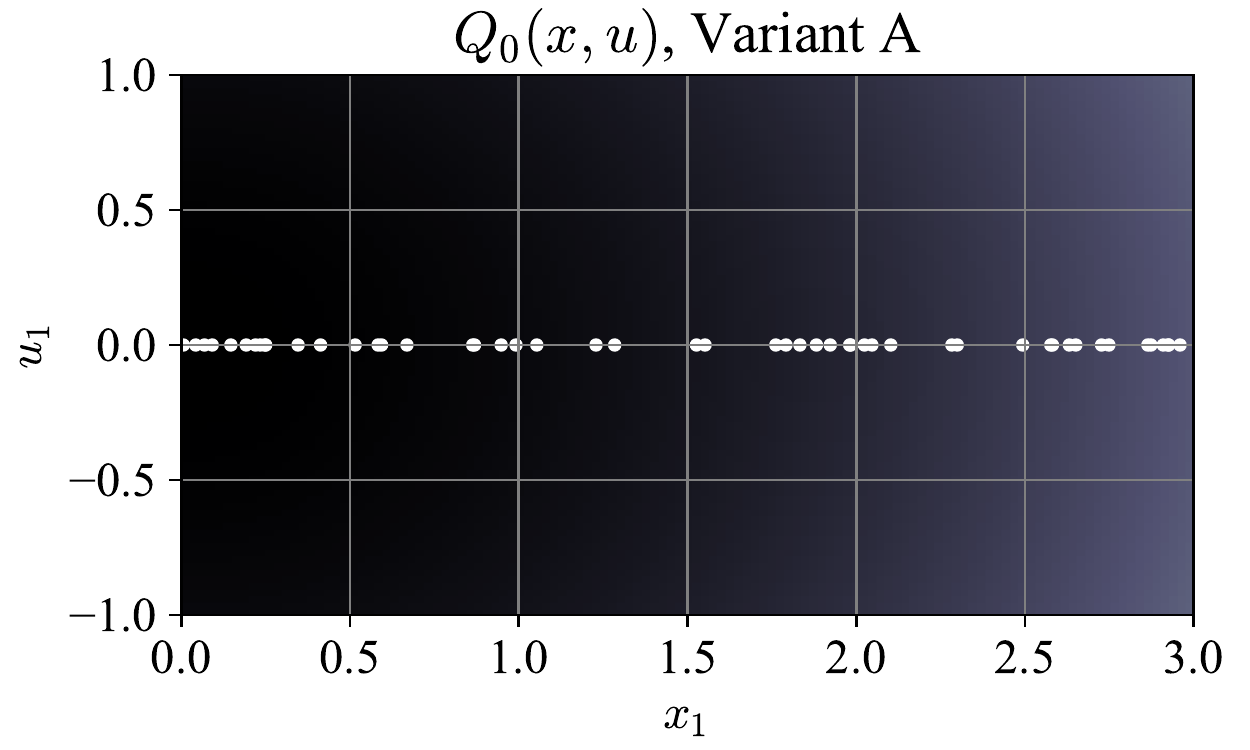}
\includegraphics[width=0.245\textwidth]{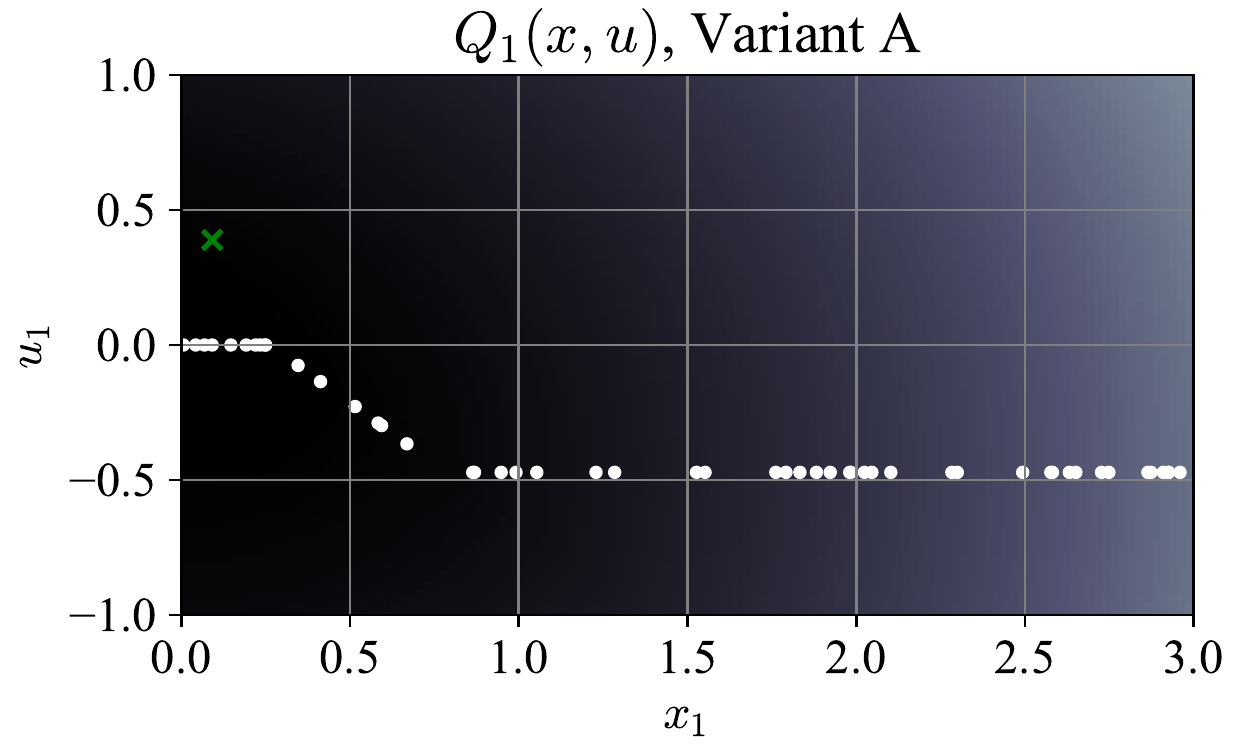}
\includegraphics[width=0.245\textwidth]{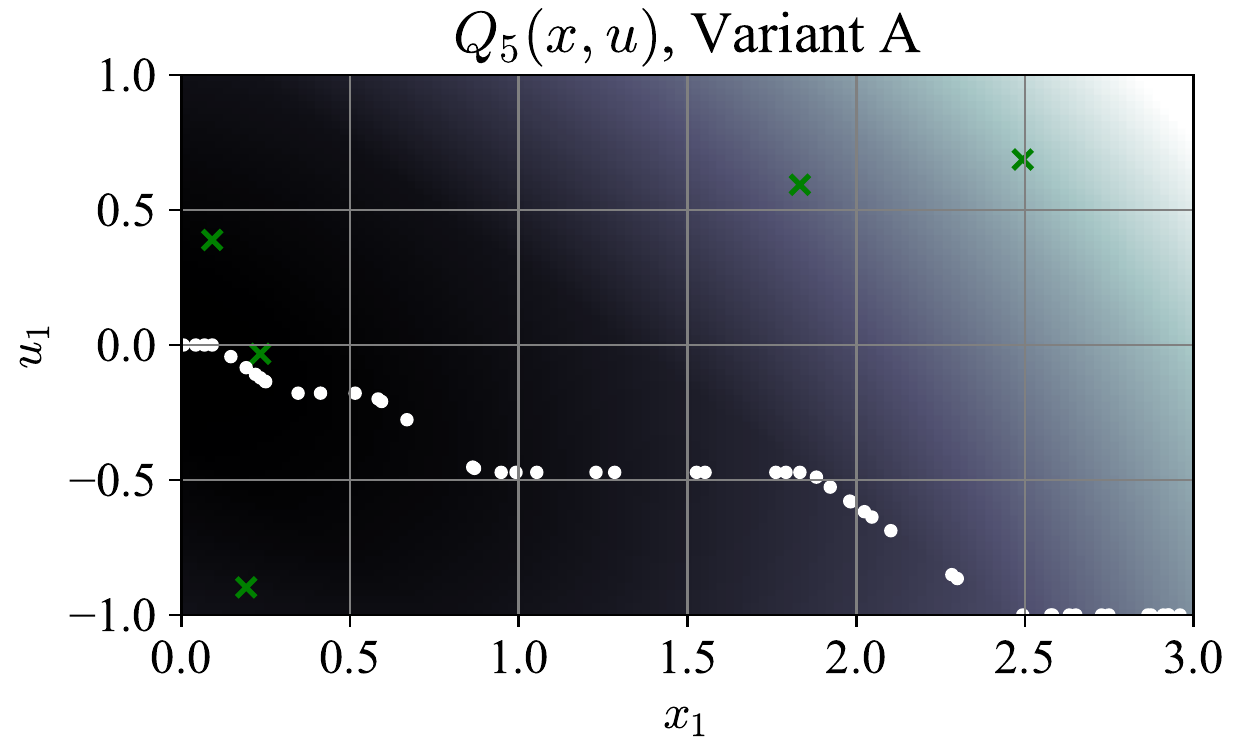}
\includegraphics[width=0.245\textwidth]{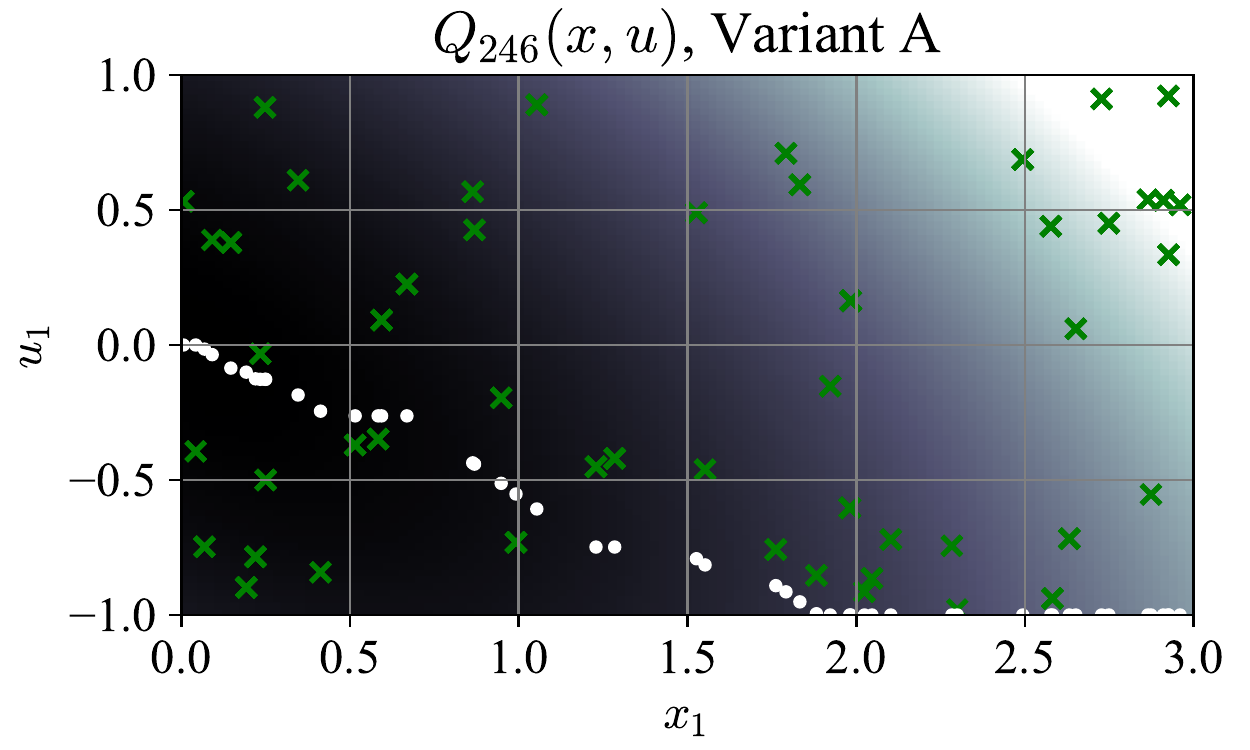}
\vspace{-0.4cm}~\hrule
~\\[0.2cm]
\includegraphics[width=0.245\textwidth]{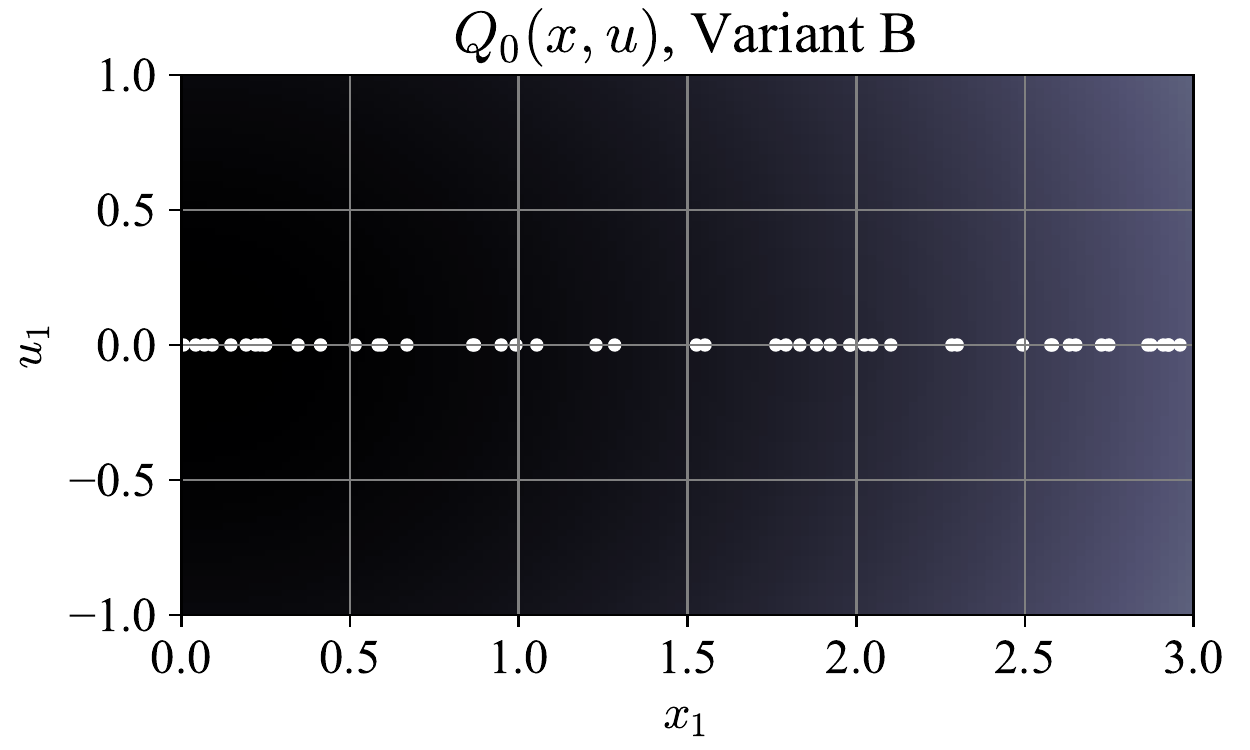}
\includegraphics[width=0.245\textwidth]{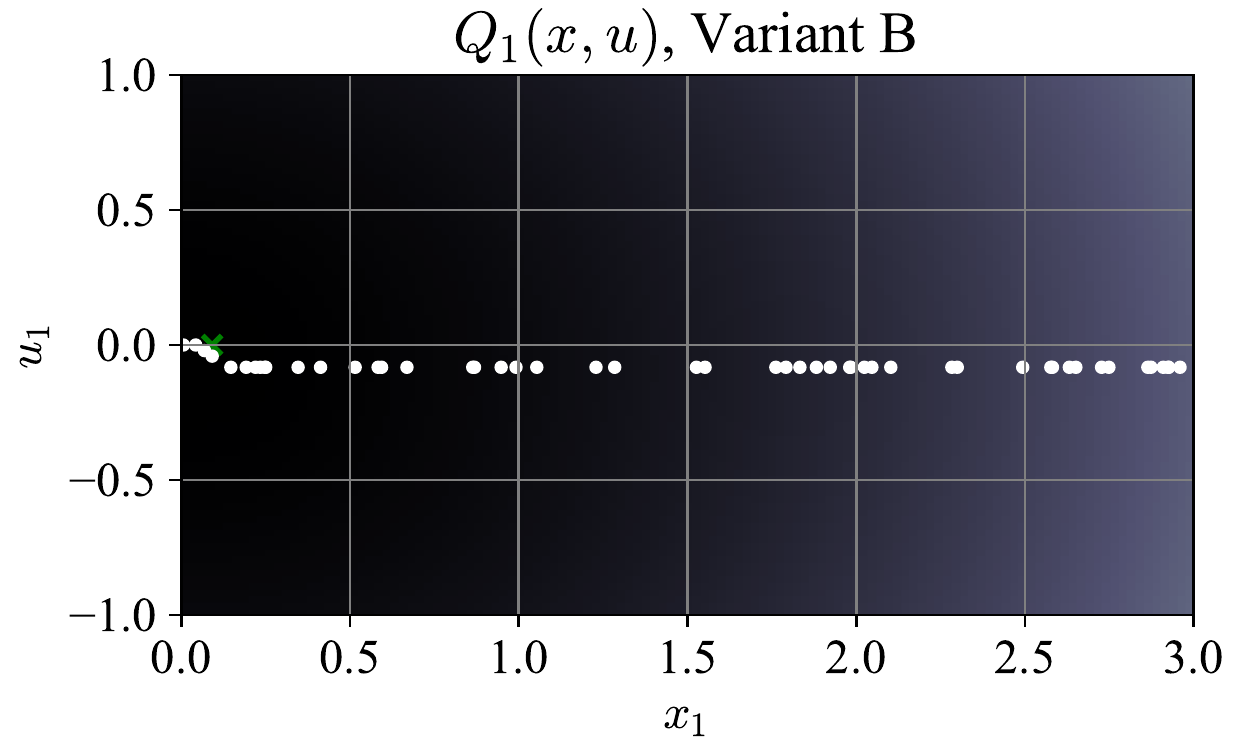}
\includegraphics[width=0.245\textwidth]{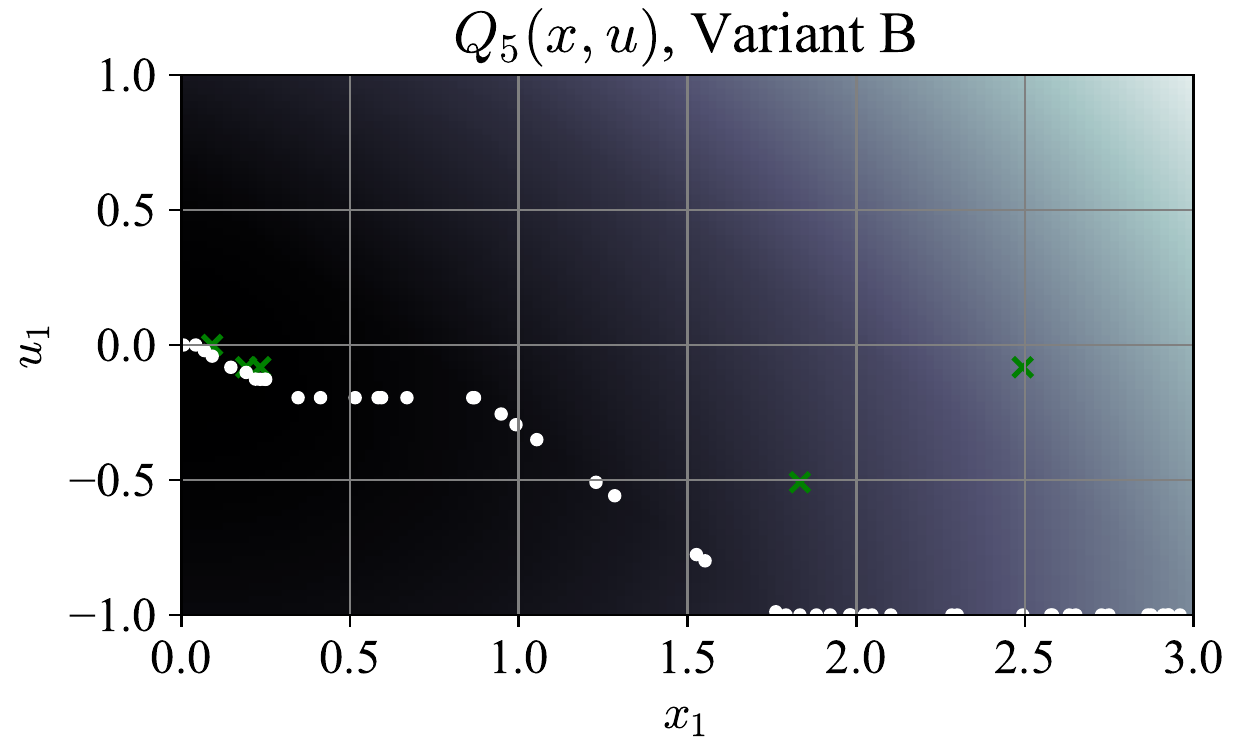}
\includegraphics[width=0.245\textwidth]{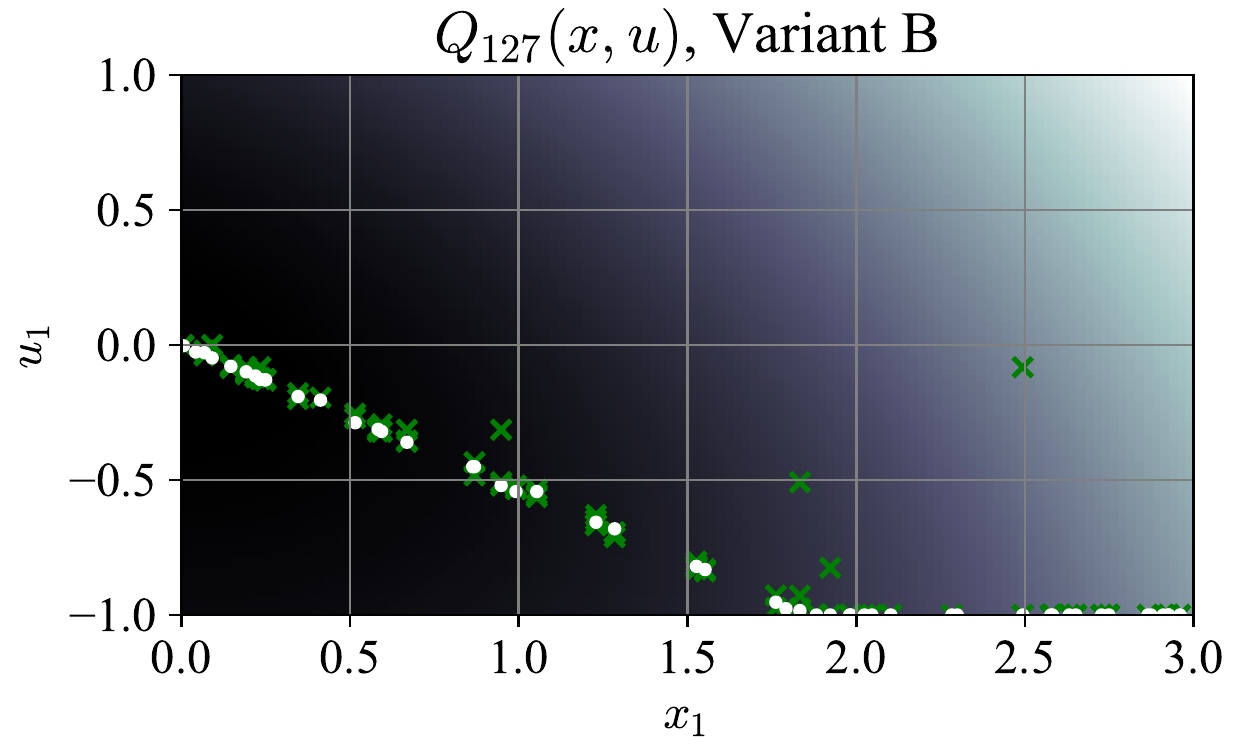}
~\vspace{-0.7cm}
\end{center}
\caption{Evolution of $Q_I(x,u)$ under Variant A (top) and Variant B (bottom) of Algorithm \ref{alg:QBenders}, for a 1-dimensional system. Lighter shading indicates a higher function value. Green markers show points $(x_m, u_m)$ cumulatively visited during iterations. White points show the policy $u = \pi(x_m; Q_I)$ for each $x_m$ in \zalg (Variant A) or \xalg (Variant B). The last plot in each row shows the last iteration, on which the termination criterion was satisfied. The policy returned by Variant B is closer to the optimal piecewise affine policy, and the $(x,u)$ points visited concentrate around this policy in the limit.}
\label{fig:Qplots}
\vspace{-0.2cm}
\end{figure*}

\subsubsection*{Higher-dimensional systems}

We tested Variant B for systems too large for the optimal policy to be computed exactly. 20 random 8-state, 3-input linear systems were created with $\|A\| \leq 0.99$. For each system, $M$ points in \xalg were sampled from a normal distribution with zero mean and variance $25$ times the identity matrix, for $M \in \{10, 20, 50, 100, 200, 500\}$. Table \ref{tab:larger_system} reports statistics upon termination. The number of iterations is roughly linear in $M$, while computation time is roughly quadratic. The latter excludes the Bellman error measurement in line \ref{algl:BellGapEnd}, on the basis that in practice, the convergence check for which it is used need not be carried out at every iteration. 

It is likely that other ways of choosing $m$ in line \ref{algl:Pickx}, e.g.~largest Bellman error, would reduce the number of iterations required, although the assumptions under which finite convergence can be guaranteed may differ. Nevertheless, total times are already modest, and we note that alternative ``exact'' DP approaches such as value iteration \cite[Ch.~2]{busoniu_reinforcement_2010} and explicit MPC \cite{borrelli_predictive_2017} are impractically expensive for problems of this size. 

\begin{table}[tbp]
\caption{Performance statistics (mean $\pm$ standard deviation) for Variant B of Algorithm \ref{alg:QBenders}; 20 random 8-state, 3-input systems.}
\begin{center}
\begin{tabular}{|c|c|c|c|}
\hline
    & Iterations to   & Computation     & Number of  \\
$M$ & termination     & time (s) & cuts added ${}^1$ \\
\hline
\hline
10 & 217.3 $\pm$ 57.9 & 0.313 $\pm$ 0.120 & 160.8 $\pm$ 31.7 \\
20 & 485.7 $\pm$ 74.5 & 1.054 $\pm$ 0.296 & 356.0 $\pm$ 71.0 \\
50 & 1293 $\pm$ 193 & 6.083 $\pm$ 1.724 & 907.0 $\pm$ 163.4 \\
100 & 2736 $\pm$ 491 & 25.09 $\pm$ 8.60 & 1828 $\pm$ 367 \\
200 & 6043 $\pm$ 1017 & 114.7 $\pm$ 34.5 & 3696 $\pm$ 626 \\
500 & 16245 $\pm$ 2527 & 777.1 $\pm$ 215.3 & 9654 $\pm$ 1644 \\
\hline
\end{tabular}
~\\[0.1cm]\footnotesize{${}^1$Cut was not added if $\varepsilon(x_m, \pi(x_m; Q_I); Q_I) < 10^{-5}$ when $x_m$ selected}
\end{center}
\label{tab:larger_system}
\vspace{-0.5cm}
\end{table}%

%%%%%%%%%%%%%%%%%%%%%%%%%%%%%%%%%%%%%%%%%%%%%%%%%%%%%%%%%%%%%%%%%%%%%%

\section{Conclusion} \label{sec:Conclusion}

This paper presented a general algorithm able to learn $Q$-functions, in the sense of minimizing Bellman error at arbitrary state space locations, for infinite-horizon problems. Convergence results were provided, both for fixed pairs $(x,u)$ and for ``policy-driven'' pairs $(x, \pi(x))$. A further variant of Algorithm \ref{alg:QBenders} could augment \xalg with sequences of states $\{x_\tau \}$ following the policy at each iteration, i.e., $x_{\tau+1} = f(x_\tau, \pi(x_\tau; Q_I))$. This would potentially learn a $Q$-function that approaches $Q^\star$ around entire trajectories, which is stronger than minimizing $\varepsilon(x_m, \pi(x_m; Q_I); Q_I)$ in individual locations $x_m \in \xalgm$.

An added attraction of our formulation is that in many cases problem \eqref{eq:Qpolicy} remains convex even when the $Q$-function is not convex in the state. Future work will investigate such situations, and consider an extension to stochastic systems.

\section*{Acknowledgement}

The author thanks Rahul Jain of the University of Southern California for valuable discussions on the topic of this paper.

%%%%%%%%%%%%%%%%%%%%%%%%%%%%%%%%%%%%%%%%%%%%%%%%%%%%%%%%%%%%%%%%%
\bibliographystyle{abbrvnamed}
\bibliography{warrington}

\begin{appendix}
\section{Examples for Assumption \ref{ass:StrongDuality}}

An example of a class of problems where Assumptions \ref{ass:CompactU} and \ref{ass:StrongDuality} hold is that which we refer to as \textit{\textbf{C-LQR}}, for which:
\begin{itemize}
\item $f(x,u) = Ax + Bu$;
\item $\ell(x,u) = \tfrac{1}{2}x^\top Q x + \tfrac{1}{2} u^\top Ru$, with $Q,R \succeq 0$;
\item $h(x,u) = Dx + Eu - \bar{h}$, defining decoupled state and input constraints, where the latter are compact; 
\item $\gamma ||A|| < 1$, meaning ``discounted-asymptotically'' stable.
\end{itemize}
\begin{prop}
Any problem of class \textbf{C-LQR} satisfies Assumptions \ref{ass:CompactU} and \ref{ass:StrongDuality}.
\end{prop}
\begin{proof}
Assumption \ref{ass:CompactU} is satisfied trivially. The lower-bounding functions in constraint \eqref{eq:osp_epi} have the quadratic form 
\begin{equation} \label{eq:CLQRq}
q_i(x,u) = \tfrac{1}{2}x^\top Q x + \tfrac{1}{2} u^\top Ru + \nu_i^\top(Ax+Bu) + \xi_i \, ,
\end{equation}
and thus the problem remains convex at each iteration $I$. A Slater point exists, namely any feasible $(x', u')$ together with any $\alpha > q_i(x',u') \, \forall i$. Thus the strong duality condition in Assumption \ref{ass:StrongDuality} holds.

To prove Lipschitz continuity in Assumption \ref{ass:StrongDuality}, one must bound the gradient in $u$-space of the functions $q_i(x_m, \cdot)$ for any given $x_m$. Inspection of problem \eqref{eq:osp_primal} shows that each new function $q_{I+1}$ depends on the existing functions $q_0, \ldots, q_I$, and \eqref{eq:CLQRq} shows that, due to $u$-compactness, a Lipschitz constant exists if the sequence $\{\|\nu_I\|\}_{I=0}^\infty$ is bounded. The KKT optimality conditions of \eqref{eq:osp_primal} include the stationarity equations
\begin{align*}
\nu_{I+1} & = D^\top \lambda_{c,I+1} + {\smash \sum_{i=0}^I} \lambda_{\alpha,i} (Qx' + A^\top \nu_i) \, ,\\
0 & = E^\top \lambda_{c,I+1} + {\smash \sum_{i=0}^I} \lambda_{\alpha,i} (Ru' + B^\top \nu_i) \, ,\\
\gamma & = {\smash \sum_{i=0}^I \lambda_{\alpha,i}} \, .
\end{align*}

Without loss of generality, one can redefine the system with a linearly scaled input, $B \rightarrow \tilde{B}$ and $R \rightarrow \tilde{R}$, such that $\|\tilde{B}\| = 1$, and then linearly scale the input constraints in $h(x,u)$ such that $\|E\| = 1$. Triangle inequalities then yield
\begin{align*}
\|\nu_{I+1}\| & \leq \|D\| \cdot \|\lambda_{c,I+1}\| + \gamma \|Q\|\cdot \|x'\| \\
& \hspace{3.4cm}+ \|A\| \cdot \left\|{\smash \sum_{i=0}^I} \lambda_{\alpha,i} \nu_i\right\| \, ,\\ 
\|\lambda_{c,I+1}\| & \leq \gamma \|\tilde{R}\| \cdot \|u'\| + \left\|{\smash \sum_{i=0}^I} \lambda_{\alpha,i} \nu_i\right\| \, .
\end{align*}
As $\mathcal{U}(x_m)$ is compact and $x_m$ is fixed, the norms of $x' = Ax_m + Bu$ and $u'$ are both bounded by some constants $X$ and $U$ respectively. Eliminating $\|\lambda_{c,I+1}\|$, one obtains
\begin{align*}
\|\nu_{I+1}\| & \leq \|D\| \left( \gamma U \|\tilde{R}\| + \left\|{\smash \sum_{i=0}^I} \lambda_{\alpha,i} \nu_i\right\|\right )  \\
& \hspace{1.8cm}+ \gamma X \|Q\| + \|A\| \cdot \|{\smash \sum_{i=0}^I} \lambda_{\alpha,i} \nu_i\| \\
& = \gamma(U\|D\|\cdot\|\tilde{R}\| + X\|Q\|) \\
& \hspace{1.8cm}+ (\|D\| + \|A\|)\left\|{\smash \sum_{i=0}^I} \lambda_{\alpha,i} \nu_i\right\| \\
& \leq \gamma(U\|D\|\cdot\|\tilde{R}\| + X\|Q\|) \\
& \hspace{1.8cm}+\gamma (\|D\| + \|A\|)\max_{i=0,...,I}\|\nu_i\| \, .
\end{align*}
Thus, $\{\|\nu_I\|\}_{I=0}^\infty$ can grow no larger than $\frac{\gamma(U\|D\|\cdot\|\tilde{R}\| + X\|Q\|)}{1 - \gamma (\|D\| + \|A\|)}$. As the state and input constraints are decoupled, $\|D\|$ can be made arbitrarily small by scaling the relevant rows of $h(x,u)$. Thus the denominator can be made strictly positive, and functions $q_i$ of the form \eqref{eq:CLQRq} are Lipschitz continuous.
\end{proof}

\end{appendix}

\end{document}